\documentclass[10pt]{amsart}

\usepackage{graphicx}
\usepackage{amsmath}
\usepackage{amssymb}

\newtheorem{theorem}{Theorem}[section]
\newtheorem{proposition}{Proposition}[section]
\newtheorem{lemma}{Lemma}[section]

\theoremstyle{definition}

\newtheorem{example}[theorem]{Example}

\theoremstyle{remark}
\newtheorem{remark}[theorem]{Remark}
\newtheorem{corollary}[theorem]{Corollary}

\numberwithin{equation}{section}

\newcommand{\ldim}{\underline{\dim}_{B}}
\newcommand{\udim}{\overline{\dim}_{B}}

\begin{document}

\title[Dimension theory approach to the complexity of a.p. trajectories]
{Dimension theory approach to the complexity of almost periodic trajectories}

\author{Mikhail Anikushin}
\address{Department of
	Applied Cybernetics, Faculty of Mathematics and Mechanics,
	Saint-Petersburg State University, Saint-Petersburg, Russia.}
\email{demolishka@gmail.com}
\thanks{This work is supported by the German-Russian
	Interdisciplinary Science Center (G-RISC) funded by the German Federal
	Foreign Office via the German Academic Exchange Service (DAAD) (Proposal M-2017a-5).}



\date{June, 2017 and, in revised form, July, 2017.}

\dedicatory{In memory of Professor V. V. Zhikov}


\begin{abstract}
We introduce and study a dimensional-like characteristic of an uniformly almost periodic function, which we call the Diophantine dimension. By definition, it is the exponent in the asymptotic behavior of the inclusion length. Diophantine dimension is connected with recurrent and ergodic properties of an almost periodic function. We get some estimates of the Diophantine dimension for certain quasiperiodic functions and present methods to investigate such a characteristic for almost periodic trajectories of evolution equations. Also we discuss the link between the presented approach and the so called effective versions of the Kronecker theorem.
\end{abstract}

\maketitle

\specialsection*{INTRODUCTION}
Methods to investigate dimension-like properties (for example, fractal or Hausdorff dimension) for almost periodic solutions are badly developed\footnote{Here we mean the study of such properties only for the solutions themselves, and not for entire attractors. There are well-known papers like \cite{Chepyzhov1994} where the fractal or Hausdorff dimension of non-autonomous attractors are estimated and in various examples the right-hand side is considered to be almost periodic.}. The starting point in dimension theory of almost periodic functions is the paper of M. L. Cartwright \cite{Cartwright1967} where she studied a link between the topological dimension and the number of frequencies of an almost periodic flow. An extension of such an approach for certain delay differential equations and partial differential equations was performed by J. Mallet-Paret in \cite{Mallet-Parret1976}. Later, K. Naito studied the fractal dimension of abstract almost periodic orbits \cite{Naito1996} and almost periodic attractors of a reaction diffusion system \cite{Naito1997}, assuming some Diophantine conditions on the frequencies. These conditions make it possible to give an upper bound for the inclusion length of almost periods (which defines the "almost periodicity" property). In terms of the inclusion length an upper estimate of the fractal dimension of an almost periodic orbit can be given. But actually, the fractal dimension does not depend on the Diophantine properties of the frequencies. Despite this, the inclusion length gives rise to the notion of a recurrent dimension, which we call the \textit{Diophantine dimension}. We note here, that later papers of K. Naito are dedicated to the study of recurrent dimensions of discrete in time almost periodic orbits with frequencies satisfying some approximating properties similar to the Diophantine condition (see next sections). And we have to note, that the definition of Diophantine dimension is contained in one of his papers \cite{Naito2001} (there it is called periodically recurrent dimension), but he did not study the properties of the Diophantine dimension explicitly. However, ideas presented in papers \cite{Naito1982} and \cite{Naito1996} are very useful after some generalizations we present here. We do not focus on some sort of estimations of the Diophantine dimension (expect simple theorem \ref{KnaitoOnefreqEstTheorem}) or another recurrent dimension, assuming some approximating conditions of exponents, but we follow a more important problem: \textit{how the recurrent properties of an almost periodic solution, provided by an almost periodic perturbation, depend on the recurrent properties of this perturbation}. In particular, one can look for a link between the Diophantine dimension of a solution and the same characteristic of a perturbation term. Such an approach is similar to methods of dimension theory. Different branches of the modern dimension theory in dynamical systems are outlined in the books \cite{Boichenko2005,Pesin2008,Robinson2010, LeoKuzReit2017,Robinson2001}.

The above problem is not new (see \cite{Kloeden2010} and links therein). Many methods to prove the existence of almost periodic solutions in addition to the existence often lead to a modules containment for a solution and a given almost periodic perturbation term\footnote{The module of an almost periodic function is the least additive subgroup of reals containing the Fourier exponents.}. It means that (in some sense) the solution is no more complicated than the perturbation term. The same can be said about the solution complexity with respect to the perturbation term if some relations between their Diophantine dimensions are proved (as in corollary \ref{coroll:KNAITO}). Moreover, to show such relations it is usually necessary to prove the containment of the sets of almost periods (theorem \ref{KNaitoTheorem} and corollary \ref{coroll:KNAITO}), which is equivalent to the module containment (theorem 4.5 in \cite{Fink2006}). Thus, relations between Diophantine dimensions of two almost periodic functions may sometimes be considered as an effective version of their modules containment.

This paper is organized as follows. At first (section 1) we give basic notions in the theory of uniformly almost periodic functions (\cite{Levitan1953,LevitanZhikov1982,Pankov2012}), topological and fractal dimension (\cite{LeoKuzReit2017,Boichenko2005,Robinson2010}), continued fractions (\cite{Khinchin1997,Schmidt1980,Katok2003}). In section 2 we describe the Liouville phenomenon for almost periodic functions and study some basic properties of the Diophantine dimension. In section 3 we estimate the Diophantine dimension for certain quasiperiodic functions (theorems \ref{HullDimensionsTh} and \ref{KnaitoOnefreqEstTheorem}). In section 4 we reformulate the result of K. Naito (theorem \ref{KNaitoTheorem} and corollary \ref{coroll:KNAITO}) for evolution equations with a strongly monotone operator as a method to estimate the Diophantine dimension of almost periodic trajectories. Section 5 is devoted to some short remarks, including the discussion of a link between the presented approach and the so called effective versions of the Kronecker theorem.
\section{Preliminaries}

\subsection{Almost periodic functions}
Let $E$ be a Banach space\footnote{To simplify statements and proofs in which the Fourier series is used we always consider Banach spaces over $\mathbb{C}$. The real case is treated similarly.}. A continuous function $u \colon \mathbb{R} \to E$ is called \textit{uniformly $E$-almost periodic} (for the sake of brevity, $E$-almost periodic or, simply, almost periodic) if for every $\varepsilon>0$ there exists $L(\varepsilon)>0$ such that for all $a \in \mathbb{R}$ there exists $\tau \in [a,a+L(\varepsilon)]$ satisfying the inequality
\begin{equation}
\label{almostperiod}
\sup_{t \in \mathbb{R}}|u(t+\tau) - u(t)|_{E} \leq \varepsilon.
\end{equation}
Here $\tau$ is called an $\varepsilon$-\textit{almost period}, and the number $L(\varepsilon)$ is the \textit{inclusion length} for $\varepsilon$-\textit{almost periods}. Denote by $l_{u}(\varepsilon)$ the minimal inclusion length for $\varepsilon$-almost periods of $u$.

For every $E$-almost periodic function $u$ there is a formal \textit{Fourier series}
\begin{equation}
\label{FourierSeres}
u(t) \sim \sum_{k=1}^{\infty}U_{k} e^{i\lambda_{k} t}
\end{equation}
with $\lambda_{k} \in \mathbb{R}$ and $U_{k} \in E$. We denote the set of all Fourier exponents $\{\lambda_{k}\}$ of $u$ by $\Lambda(u)$.

\begin{theorem}[Approximation theorem]
	\label{th: Approximation-th}
	Every uniformly almost periodic function (\ref{FourierSeres}) is the uniform (on $\mathbb{R}$) limit of a sequence of trigonometric polynomials
	\begin{equation}
	P_{\varepsilon}(t) = \sum\limits_{k=1}^{n_{\varepsilon}}B_{k}^{\varepsilon} e^{i \lambda_k t}.
	\end{equation}
\end{theorem}

The \textit{hull} of $u$, $\mathcal{H}(u)$, is defined by the set
\begin{equation}
\mathcal{H}(u) := Cl\{u_{\tau} (\cdot) := u (\cdot + \tau)\ | \ \tau \in \mathbb{R} \},
\end{equation}
where the closure is taken in the topology of uniform convergence in the space $C_{b}(\mathbb{R}; E)$. The hull of an almost periodic function is a compact minimal set, i.e. for every $v \in \mathcal{H}(u)$ we have that $v$ is almost periodic and $\mathcal{H}(v)=\mathcal{H}(u)$.

\subsection{Topological and fractal dimension}

For a given metric space $(X,\rho)$ we denote its Lebesgue covering dimension by $\dim_{T}X$. In the further we will deal with compact (and, consequently, separable) metric spaces, therefore $\dim_{T}X$ we call simply a topological dimension of $X$ and this will not cause misunderstandings.

Now let $(X,\rho)$ be a compact metric space and $N_{\varepsilon}(X)$ is the smallest number of open balls with radius $\varepsilon$ required to cover $X$. The limit
\begin{equation}
\ldim(X) := \liminf\limits_{\varepsilon \to 0+}\frac{\ln N_{\varepsilon}(X)}{\ln(1/\varepsilon)}
\end{equation}
is called the \textit{lower box dimension} of $X$ and the limit
\begin{equation}
\udim(X) := \limsup\limits_{\varepsilon \to 0+}\frac{\ln N_{\varepsilon}(X)}{\ln(1/\varepsilon)}
\end{equation}
is called the \textit{fractal} or \textit{upper box dimension} of $X$.
The following inequality holds
$$\dim_{T}(X) \leq \ldim(X) \leq \udim(X).$$
If $\ldim(X)$ and $\udim(X)$ coincide, we write $\dim_{F}(X)$ for this common value, which we call fractal dimension of $X$.

In a contrast to the topological dimension, the fractal dimension is not a topological invariant, i.e. its value can change if we replace the given metric by a topologically equivalent one (=generating the same topology). If we want to emphasize the choice of the metric we write $\dim_{F}(X,\rho)$.

\begin{example}
	\label{Example1}
	Let $(X,\rho)$ be a compact metric space. For an arbitrary $\alpha \in (0,1]$ we put a new metric on $X$ (which is topologically equivalent to $\rho$) by
	\footnote{Triangle inequality follows from the inequality $(x+y)^{\alpha} \leq x^{\alpha} + y^{\alpha}$ for $x,y\geq 0$ and $\alpha \in (0,1]$.}
	 $$\rho_{\alpha}(x,y) := \rho^{\alpha}(x,y), \ x,y \in X.$$
	 One can show that $\ldim(X,\rho_{\alpha})=\frac{\ldim(X,\rho)}{\alpha}$ and $\udim(X,\rho_{\alpha})=\frac{\udim(X,\rho)}{\alpha}$.
\end{example}

A metric $\rho$ is \textit{stronger} than another metric $\rho'$ if there exists a constant $C>0$ such that for all $x,y \in X$ the inequality
\begin{equation}
\label{UniformEqMetrics}
\rho'(x,y) \leq C\rho(x,y).
\end{equation}
is satisfied. The following lemma is easy to check. 

\begin{lemma}
	\label{FractalDimEstimateLemma}
	Let $(X,\rho)$ be a compact metric space and $\rho'$ is another metric on $X$ such that $\rho$ is stronger that $\rho'$; then $\udim(X,\rho') \leq \udim(X,\rho)$ and $\ldim(X,\rho') \leq \ldim(X,\rho)$.
\end{lemma}

Let $X$ be a topological space with the topology generated by any of two metrics $\rho$ and $\rho'$. Choose an arbitrary open cover $\mathcal{U}$ of $X$. We say, that $\rho$ is \textit{locally stronger} than $\rho'$ w.r.t. $\mathcal{U}$ if the inequality (\ref{UniformEqMetrics}) holds with a constant $C=C(U)$ for all $x,y \in U$, where $U \in \mathcal{U}$ is arbitrary. Since we do not know how well the following lemma is known, we give a proof of it.

\begin{lemma}
	\label{LocalEstimatesLemma}
	Let $X$ be a compact topological space with the topology generated by any of metrics $\rho$ and $\rho'$. Let $\mathcal{U}$ be an open cover of $X$; then the following statements are equivalent
	\begin{enumerate}
		\item $\rho$ is stronger than $\rho'$;
		\item $\rho$ is locally stronger than $\rho'$ w.r.t. $\mathcal{U}$.
	\end{enumerate}
\end{lemma}
\begin{proof}
	Let us show $2 \Rightarrow 1$. Assuming the opposite we find two sequences $x_n,y_n \in X, n=1,2,\ldots$ such that
	\begin{equation}
	\frac{\rho'(x_n,y_n)}{\rho(x_n,y_n)} \geq n.
	\end{equation}
	By compactness, there exist convergent subsequences $\{x'_n\} \subset \{x_n\}$ and $\{y'_n\} \subset \{y_n\}$ such that $x'_n \to x_0$ and $y'_n \to y_0$. Let $2\delta>0$ be the Lebesgue number of the cover $\mathcal{U}$ w.r.t. $\rho$. It is clear that $\rho(x_0,y_0)>\delta$. Thus,
	\begin{equation}
	\frac{\rho'(x_0,y_0)}{\delta} \geq n, \ n=1,2,\ldots,
	\end{equation}
	and this is a contradiction.
\end{proof}

\subsection{Continued fractions}

\label{subsec: continuedfractions}

Let $G \colon (0,1) \to (0,1)$ be the \textit{Gauss map}, defined by the equality
\begin{equation}
G(x):=\frac{1}{x} - \left\lfloor\frac{1}{x}\right\rfloor.
\end{equation}
For every number $\omega \in \mathbb{R} \setminus \mathbb{Q}$ we consider its fraction expansion $\{ a_{k} \}_{k \geq 0}$ defined as follows. Firstly, we define $\{ \omega_{k} \}_{k \geq 0}$. Let $\omega_0 = \omega - \left\lfloor\omega\right\rfloor$ and $\omega_{k}=G^{k}(\omega_0)$ for $k \geq 1$. Then the terms of the continued fraction of $\omega$ is defined by $a_{0} = \left\lfloor\omega\right\rfloor$ and $a_{k} =\left\lfloor\frac{1}{\omega_{k-1}}\right\rfloor$ for $k\geq 1$. So, we write (formally)
\begin{equation}
\omega = [a_0; a_1, a_2, a_3,\cdots] = a_0+\cfrac{1}{a_1+\cfrac{1}{a_2+\cfrac{1}{a_3+\ldots}}}\;
\end{equation}
For $k\geq 0$ the fraction
\begin{equation}
\frac{p_k}{q_k} = a_0 + \cfrac{1}{a_1 + \cfrac{1}{a_2 + \cfrac{1}{ \ddots + \cfrac{1}{a_k} }}}
\end{equation}
is called the \textit{$k$-th convergent} of $\omega$. The following estimate holds
\begin{equation}
\label{ConFracIneq}
\frac{1}{q_{k}(q_{k+1}+q_{k})}<\left|\omega - \frac{p_{k}}{q_k} \right| < \frac{1}{q_{k+1}q_k} < \frac{1}{a_{k+1} q^2_k}.
\end{equation}
Let $p_{-2}=0$,$p_{-1}=1$,$q_{-2}=1$,$q_{-1}=0$. Then the convergents satisfy the following recurrence relation 
\begin{equation}
\label{ConvEq1}
p_{k}=a_{k}p_{k-1}	+ p_{k-2},
\end{equation}
\begin{equation}
\label{ConvEq2}
q_{k}=a_{k}q_{k-1}	+ q_{k-2}.
\end{equation}

The convergents (for $k\geq 1$) give the best approximations of an irrational number, i.e. the difference $|q_k\omega-p_k|$  is the minimal among all the differences like $|q\omega-p|$, where $q$ is lesser than $q_{k+1}$. From (\ref{ConFracIneq}) it is clear that the term $a_{k+1}$ determines the quality of approximation by the $k$-th convergent $\frac{p_k}{q_k}$. Thus, the growth rate of the sequence $a_{k}$ determines the quality of approximation of an irrational number by rationals. There are many phenomenons in the theory of dynamical systems related to such approximation properties (see \cite{Katok1997} or \cite{Katok2003}).
\section{Diophantine dimension}

\subsection{Liouville phenomenon for almost periodic functions}
\label{subs:Liouvillephenomenon}
Consider an almost periodic trajectory given by $u(t) = e^{i2\pi t} + e^{i 2\pi \omega t}$, where $\omega$ is an irrational number. From the Kronecker theorem (theorem \ref{th: Kronecker-Th}) it follows that the trajectory is dense in the circle of radius 2 centred at the origin. Let $D=Cl(u(\mathbb{R}))$. The following proposition is easy to prove.
\begin{proposition}
	\label{ApDistribStat}
	There is a Borel measure $\mu$ supported on $D$ such that $\mu$ is independent of the irrational $\omega$ and $u$ is uniformly distributed with respect to $\mu$. In other words, for any Borel subset $C \subset D$ we have
	\begin{equation}
	\label{ApDistrib}
	\lim\limits_{T \to +\infty} \frac{1}{T}\int\limits_{0}^{T}\chi_{C}(u(t))dt=\mu(C).
	\end{equation}
\end{proposition}

Now we will choose distinct $\omega$'s, which differs by the approximation (by rationals) properties and see what happens. Let, for example, $\omega=[0;5,10^{9},\ldots]$. Then from (\ref{ConFracIneq}) it follows that $|\omega-\frac{1}{5}| < 10^{-9}$. Thus, $u(t)$ is close to $g(t) = e^{i2\pi t} + e^{i \frac{2\pi}{5} t}$ for a long time, namely,  $|f(t)-g(t)| < 10^{-3}$ for $t \in [0,10^6]$. Now we put $\omega=\sqrt{2}=[1;2,2,2,\ldots]$, which is a badly approximable number. To get a good rational approximation to $\sqrt{2}$ we need a fraction with large enough denominator. Thus, a trajectory should be more complicated, than in the previous case. This phenomenon can be observed by Figure \ref{ris:image1}.
\begin{figure}[h]
	\begin{minipage}[h]{0.49\linewidth}
		\center{\includegraphics[width=1\linewidth]{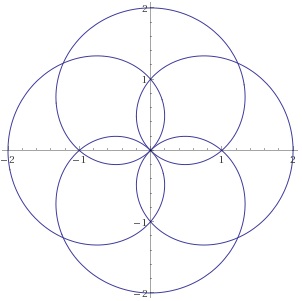} \\ a) $\omega=\frac{1}{5}, t=0..5$}
	\end{minipage}
	\hfill
	\begin{minipage}[h]{0.49\linewidth}
		\center{\includegraphics[width=1\linewidth]{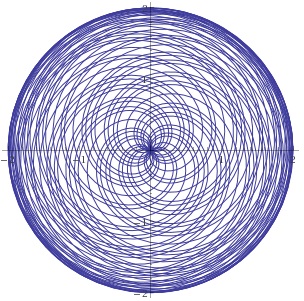} \\ b) $\omega=\sqrt{2}, t=0..50$}
	\end{minipage}
	\caption{A part of the trajectory $u(t)$ for certain $\omega$.}
	\label{ris:image1}
\end{figure}
It should be noted that such effects may appear at any level of approximation. But, by the proposition \ref{ApDistribStat}, both trajectories have the same distribution. Thus, a good approximation of $\omega$ means a slow convergence of the limit in (\ref{ApDistrib}).

By the theorem \ref{th: Approximation-th} a similar property can be shown in the case of general almost periodic function. But now there are two factors. The first is the quality of simultaneous approximation of the exponents of the polynomial $P_{\varepsilon}$ (see theorem \ref{th: Approximation-th}), and the second is the rate of convergence of the sequence $P_{\varepsilon}$ to $u$, which determined, roughly speaking, by the decay rate of the Fourier coefficients. As it can be seen that the same factors affects the distribution of the almost periods. If the exponents of an almost periodic trajectory can be approximated extremely fast then there occurs the, so called, \textit{Liouville phenomenon}. A simple version of it appears in irrational rotations of a circle (see \cite{Katok1997,Katok2003}).

\subsection{Definition and basic properties}
Let $E$ be a Banach space and $u$ is an $E$-almost periodic function (non-zero). The limit\footnote{As we noted earlier, this definition is contained in \cite{Naito2001}. We use another name and symbol to emphasize the nature and importance of the introduced object in the theory of almost periodic functions.}
\begin{equation}
\label{diophantinedimension}
\mathfrak{Di}(u):=\limsup_{\varepsilon \to 0+}\frac{\ln l_{u}(\varepsilon)}{\ln 1 / \varepsilon}
\end{equation}
is called \textit{the Diophantine dimension} of $u$.

Let $F$ be another Banach space. We say that a map $\chi \colon E \to F$ satisfies \textit{the H\"{o}lder condition}, if there are constants $\alpha \in (0,1]$ and $C>0$ such that the inequality
\begin{equation}
\label{GelderProp}
|\chi(x)-\chi(y)|_{F} \leq C |x-y|^{\alpha}_{E}
\end{equation}
holds for all $x,y \in E$. Here $\alpha$ is called \textit{the H\"{o}lder exponent}.

It is clear that $\chi \circ u$ is $F$-almost periodic.
\begin{proposition}
	\label{GelderTransLemma}
	Let $u$ be an $E$-almost periodic function and $\chi \colon E \to F$ satisfies the H\"{o}lder condition with an exponent $\alpha \in (0,1]$; then
	\begin{equation}
	\mathfrak{Di}(\chi \circ u) \leq \frac{\mathfrak{Di}(u)}{\alpha}.
	\end{equation}
\end{proposition}
\begin{proof}
	Let $\tau$ be an $\varepsilon^{\frac{1}{\alpha}}$-almost period of $u$. Then
	$$|(\chi \circ u )(t+\tau)-(\chi \circ u )(t)|_{F} \leq C|u(t+\tau)-u(t)|^{\alpha}_{E} \leq C \varepsilon, \text{ for any } t \in \mathbb{R}.$$
	In other words, any $\varepsilon^{\frac{1}{\alpha}}$-almost period of $u$ is a $C\varepsilon$-almost period of $\chi \circ u$. Thus,
	$$l_{\chi \circ u}(C\varepsilon) \leq l_{u}(\varepsilon^{\frac{1}{\alpha}}).$$
	After simple transformations, we have
	\begin{equation}
	\label{GelderTransLemmaEq}
	\frac{\ln l_{\chi \circ u}(C\varepsilon)}{\ln C + \ln 1 / C\varepsilon}=\frac{\ln l_{\chi \circ u}(C\varepsilon)}{\ln 1 / \varepsilon} \leq \frac{\ln l_{u}(\varepsilon^{\frac{1}{\alpha}})}{\ln 1 / \varepsilon} = \frac{\ln l_{u}(\varepsilon^{\frac{1}{\alpha}})}{\alpha \ln 1 / \varepsilon^{\frac{1}{\alpha}} }.
	\end{equation}
	Passing to the upper limit as $\varepsilon \to 0+$ in (\ref{GelderTransLemmaEq}), we get the statement.
\end{proof}
\begin{corollary}
	Let $\chi$ be a Bi-Lipschitz map, i.e. both $\chi$ and $\chi^{-1}$ satisfy the H\"{o}lder condition with exponent $1$; then
	\begin{equation}
\label{bilipschitz}
	\mathfrak{Di}(\chi \circ u) = \mathfrak{Di}(u).
	\end{equation}
\end{corollary}

\noindent It is easy to see that the statement of proposition \ref{GelderTransLemma} holds if the inequality (\ref{GelderProp}) is satisfied only on the range of $u$, i.e. for $x, y \in u(\mathbb{R})$. Due to the pre-compactness of $u(\mathbb{R})$, to preserve the Diophantine dimension of $u$ it is enough for $\chi$ to be bijective and continuously differentiable in the Fr\'echet sense.

We also consider \textit{the lower Diophantine dimension} defined as
\begin{equation}
\label{Lowerdiophantinedimension}
\mathfrak{di}(u):=\liminf_{\varepsilon \to 0+}\frac{\ln l_{u}(\varepsilon)}{\ln 1 / \varepsilon}
\end{equation}
It is clear that all the properties considered above hold for the lower Diophantine dimension too.
\section{Diophantine dimension of a quasiperiodic function}

An $E$-almost periodic function $u(t) \sim \sum\limits U_{k}e^{i \lambda_{k}t}$ is called \textit{quasiperiodic} if there are rationally independent real numbers $\omega_{1},\ldots,\omega_{n}$ such that for all $k=1,2,\ldots$ the expansion
\begin{equation}
	\label{QuasiperRazl}
	\lambda_k = \sum\limits_{j=1}^{n}a^{(k)}_{j} \omega_j, \text{where } a^{(k)}_{j} \in \mathbb{Z}, j=1,2,\ldots,n.
\end{equation}
holds. In this case $\omega_1,\ldots,\omega_n$ are called \textit{frequencies}. The following theorem is well-known.

\begin{theorem}
	Let $u$ be a quasiperiodic function with frequencies $\omega_1,\ldots,\omega_n$; then there is 1-periodic\footnote{That means that $h$ is 1-periodic in each coordinate or, in other words, h is periodic w.r.t. the lattice $\mathbb{Z}^{n}$.} continuous function $h(t_{1},\ldots,t_{n})$, such that 
	\begin{equation}
		\label{QuasiPerPreds}
		u(t) = h\left(\frac{\omega_{1}}{2\pi}t,\ldots,\frac{\omega_{n}}{2\pi}t\right), \ t \in \mathbb{R}.
	\end{equation}
\end{theorem}

Further, we will deal with this "representing" function $h$.

\subsection{An absolute lower bound}

Consider the $n$-dimensional flat torus $\mathbb{T}^n=\mathbb{R}^n/\mathbb{Z}^n$ with the metric defined as follows. For given $\theta', \theta'' \in \mathbb{T}^n$ we set
\begin{equation}
\rho_{\mathbb{T}^n}(\theta',\theta''):=\min_{\theta_1 \in \theta', \theta_2 \in \theta''}\|\theta_1-\theta_2\|,
\end{equation}
where $\|.\|$ is the sup-norm in $\mathbb{R}^n$. We find it convenient to write $|\theta'-\theta''|_{\mathbb{T}^n}$ instead of $\rho_{\mathbb{T}^n}(\theta',\theta'')$\footnote{Considering the flat torus as an additive group one can see that the metric $\rho_{\mathbb{T}^n}$ is translation invariant. Thus, $|\theta|_{\mathbb{T}^n}:=\rho_{\mathbb{T}^n}(\theta,0)$ is well-defined.}. Further, we do not distinguish a vector in $\mathbb{R}^n$ and its class of equivalence (=the corresponding point on $\mathbb{T}^n$) It is easy to see that $|\theta'-\theta''|_{\mathbb{T}^n} = \max\limits_{j=1,\ldots,n}|\theta'_j -\theta''_j|_{\mathbb{T}^1}$, where $\theta'=(\theta'_1,\ldots,\theta'_n)$ and $\theta''=(\theta''_1,\ldots,\theta''_n)$. Also, the equality $\dim_{T}(\mathbb{T}^n)=\dim_{F}(\mathbb{T}^n) = n$ is not hard to prove.

We will need the following theorem.
\begin{theorem}[Kronecker's theorem]
	\label{th: Kronecker-Th}
	Let $\omega = (\omega_{1},\ldots, \omega_{n})$ be an $n$-tuple of rationally independent real numbers and $\theta \in \mathbb{T}^{n}$; then the inequality
	\begin{equation}
	\label{KroneckerSystem}
	|\omega t - \theta|_{\mathbb{T}^n} < \varepsilon
	\end{equation}
	has a solution $t=t_{\varepsilon} \in \mathbb{R}$ for every $\varepsilon>0$.
\end{theorem}

If $h(t_1,\ldots,t_n)$ is a $1$-periodic function of real variables, then it is natural to consider $h$ as a function with domain $\mathbb{T}^n$. So, the value $h(t_1+\theta_1,\ldots,t_n+\theta_n)$ is obviously defined for $\theta=(\theta_1,\ldots,\theta_n) \in \mathbb{T}^n$.

A function $h \colon \mathbb{T}^n \to E$ satisfies a \textit{reverse H\"{o}lder condition} with an exponent $\alpha \in (0,1]$ at a point $\theta_{0} \in \mathbb{T}^n$ if for some constants $\varepsilon_{0}>0$ and $C>0$ such that
\begin{equation}
\label{PointHoelderCondition}
|h(\theta)-h(\theta_0)|_{E} \geq  C |\theta_0-\theta|^{\alpha}_{\mathbb{T}^{n}}
\end{equation}
provided that $|\theta_0-\theta|_{\mathbb{T}^n} \leq \varepsilon_0$. And we say that $h$ satisfies a \textit{locally H\"{o}lder condition} with an exponent $\alpha \in (0,1]$ if for some constants $\varepsilon_{0}>0$ and $C>0$ such that
\begin{equation}
|h(\theta')-h(\theta'')|_{E} \leq  C |\theta'-\theta''|^{\alpha}_{\mathbb{T}^{n}}
\end{equation}
provided that $|\theta'-\theta''|_{\mathbb{T}^{n}} \leq \varepsilon_0$.

We call a function $h \colon \mathbb{T}^{n} \to E$ \textit{strictly $1$-periodic} if the equality $h(\theta+\theta')=h(\theta)$ for all $\theta \in \mathbb{T}^n$ implies that $\theta' = 0.$

\begin{lemma}
	\label{FreqPeriodicity}
	Let $u(t)=h(\omega_1t,\ldots,\omega_n t)$ be an $E$-quasiperiodic function such that $2\pi \omega_{1},\ldots,2\pi\omega_{n} \in \Lambda(u)$; then $h$ is strictly $1$-periodic. 
\end{lemma}

\begin{proof}
	Let 
	\begin{equation}
	u(t) \sim \sum\limits_{k=1}^{\infty}U_{k} e^{i\lambda_{k} t} \sim \sum\limits_{k=1}^{\infty}U_{k} e^{i \sum\limits_{j=1}^{n} a^{(k)}_{j} 2\pi\omega_{j} t}.
	\end{equation}
	Let $\theta' \in \mathbb{T}^n$ and $h(\theta + \theta')=h(\theta)$ for all $\theta \in \mathbb{T}^n$. Consider $v(t) = h(\omega_1 t + \theta'_1, \ldots, \omega_n t + \theta'_n)$. By the Kronecker theorem there is a sequence $t_{k} \in \mathbb{R}, k=1,2,\ldots$ such that $\omega_{j} t_{k} \to \theta'_{j} \pmod 1$ and, consequently, $u_{t_k} \to v(t)$ uniformly on $\mathbb{R}$. In particular, $v(t)$ is almost periodic and
	\begin{equation}
	v(t) \sim \sum\limits_{k=1}^{\infty}U_{k} e^{\sum\limits_{j=1}^{n} a^{(k)}_j 2\pi \theta'_{j}} e^{i\lambda_{k} t}.
	\end{equation}
	From the equality $u(t)=v(t)$ for all $t \in \mathbb{R}$ and the uniqueness theorem we have that for all $k=1,2,\ldots$
	\begin{equation}
	U_{k} = U_{k} e^{\sum\limits_{j=1}^{n}a^{(k)}_{j} 2\pi \theta'_{j}}.
	\end{equation}
	In other words, $\sum\limits_{j=1}^{n}a^{(k)}_j 2\pi \theta'_j = 0 \pmod{2\pi}$ or $\sum\limits_{j=1}^{n}a^{(k)}_j \theta'_j = 0 \pmod 1$. Since there are $k_{1},\ldots,k_{n}$ such that $2\pi\omega_{j} = \lambda_{k_{j}}$ for $j=1,\ldots,n$ we have that $\theta'_{j}=0$ for all $j=1,\ldots,n$.
\end{proof}

\begin{theorem}
	\label{HullDimensionsTh}
	Let $u(t)=h(\omega_{1}t,\ldots,\omega_{n} t)$ be an $E$-quasiperiodic function such that $h$ is strictly $1$-periodic; then
	\begin{enumerate}
		\item  $\dim_{T}\mathcal{H}(u)=n$;
		\item If $h$ satisfies the locally H\"{o}lder condition with an exponent $\alpha$, then $\udim(\mathcal{H}(u)) \leq \frac{n}{\alpha}$;
		\item If $h$ satisfies the reverse H\"{o}lder condition with an exponent $\alpha$ at a point $\theta_{0}$, then $\ldim(\mathcal{H}(u)) \geq \frac{n}{\alpha}$.
	\end{enumerate}
\end{theorem}

\begin{proof}
	Let $v \in \mathcal{H}(u) $. Then there is a sequence $t_{k}$, $k=1,2,\ldots$, $u_{t_{k}} \to v$ uniformly on $\mathbb{R}$ as $k\to +\infty$. One can find a subsequence $t'_{k} \subset t_{k}$ such that for some $\theta_{1},\ldots,\theta_{n} \in [0,1)$ (or $\theta=(\theta_{1},\ldots,\theta_{n}) \in \mathbb{T}^n$) the following
	\begin{equation}
	\label{hulldesclimits}
	\omega_{j} t'_{k} \to \theta_j \pmod{1}, \ \ (j=1,\ldots,n).
	\end{equation}
	holds. It is clear that $v(t) = h(\omega_{1} t + \theta_{1},\ldots,\omega_{n} t + \theta_{n})$. Consider the map $\chi \colon \mathcal{H}(u) \to \mathbb{T}^n$ such that $\chi(v) := (\theta_1,\ldots,\theta_n)$, where $\theta_{1},\ldots,\theta_{n}$ defined in (\ref{hulldesclimits})\footnote{Due to strictly $1$-periodicity the numbers $\theta_{1},\ldots,\theta_{n}$ are determined uniquely.}. By the Kronecker theorem, $\chi$ is a surjection and, as a continuous bijective map between compact metric spaces, is a homeomorphism. So, we get item (1) of the theorem.
	
	Consider a new metric on $\mathbb{T}^n$, which is induced by $\chi$, i.e. for $\theta', \theta'' \in \mathbb{T}^n$ we set
	\begin{equation}
	\label{inducedtorusmetric}
	\rho'(\theta', \theta''):= \sup_{t \in \mathbb{R}} |h(\omega_{1} t + \theta'_{1},\ldots,\omega_{n} t + \theta'_{n}) - h(\omega_{1} t + \theta''_{1},\ldots,\omega_{n} t + \theta''_{n})|_{E}.
	\end{equation}
	\noindent	By the Kronecker theorem and from the continuity of $h$, we have
	\begin{equation}
	\label{FinalInducedMetric}
	\rho'(\theta', \theta'') = \sup_{\theta \in \mathbb{T}^n}|h(\theta+\theta') - h(\theta+\theta'')|_{E}.
	\end{equation}
	
	So, within the condition of item (2) of the theorem, we have $\rho'(\theta', \theta'') \leq C |\theta'-\theta''|^{\alpha}_{\mathbb{T}^{n}}$
	provided by $|\theta'-\theta''|_{\mathbb{T}^{n}} \leq \varepsilon_0$.
	In other words, the metric $\rho^{\alpha}_{\mathbb{T}^n}$ is stronger than the metric $\rho'$ with respect to a cover generated by open balls of radius  $\frac{\varepsilon_0}{2}$ in the metric $\rho_{\mathbb{T}^n}$. Using lemma \ref{LocalEstimatesLemma}, example \ref{Example1} and lemma \ref{FractalDimEstimateLemma}, we finish the proof of item (2).
	
	Within the conditions of item (3) we have $\rho'(\theta', \theta'') \geq C |\theta'-\theta''|^{\alpha}_{\mathbb{T}^{n}}$\footnote{To get this one has to put $\theta = \theta_0 - \theta'$ in (\ref{FinalInducedMetric}).} provided by $|\theta'-\theta''|_{\mathbb{T}^{n}} \leq \varepsilon_{0}$. The further is analogous to the proof of item (2).
\end{proof}  

\begin{remark}
	Item (1) of the theorem \ref{HullDimensionsTh} is a special case of the results of M. L. Cartwright (see theorem 8 in \cite{Cartwright1967}). Parts (2) and (3) is a generalization of results from \cite{Anikushin2016} (see corollary \ref{PolynomialLowerEstimate}). The following lemma is an easy generalization of lemma 2 in \cite{Anikushin2016}, the same idea is contained in the proof of theorem 1 in \cite{Naito1996}.
\end{remark}

\begin{proposition}
	Let $u$ be an $E$-almost periodic function. For each $\varepsilon>0$ let $\delta(\varepsilon)>0$ be any number such that $|u(t)-u(s)|_{E} \leq \varepsilon$ provided by $|t-s|\leq \delta(\varepsilon)$; then
	\begin{equation}
	\label{LowerDimEstimate}
	\ldim(\mathcal{H}(u)) \leq \mathfrak{di}(u) + \liminf\limits_{\varepsilon \to 0+}\frac{\ln 1 / \delta(\varepsilon)}{\ln 1 /\varepsilon}.
	\end{equation}
\end{proposition}

\begin{corollary}
	Let $u=h(\omega_{1} t,\ldots,\omega_{n} t)$ be a quasiperiodic function and $h$ satisfies the conditions of items (2) and (3) of theorem \ref{HullDimensionsTh}; then
	\begin{equation}
	\mathfrak{di}(u) \geq \frac{n-1}{\alpha}.
	\end{equation}
\end{corollary}

\begin{proof}
	For sufficiently small $\varepsilon>0$ we put $\delta(\varepsilon):=C\varepsilon^{\frac{1}{\alpha}}$. Using (\ref{LowerDimEstimate}) we finish the proof.
\end{proof}

\begin{corollary}
	\label{PolynomialLowerEstimate}
	Let $P(t) = \sum\limits_{k=1}^{n}A_{k} e^{i 2\pi \omega_{k} t}$ be a trigonometric polynomial, where $A_{1},\ldots,A_{n}$ are non zero vectors in $E$ and $\omega_{1},\ldots,\omega_{n}$ are rationally independent; then $\mathfrak{di}(P) \geq n-1.$
\end{corollary}
\begin{proof}
We will show that the metric $\rho'$ given by (\ref{inducedtorusmetric}) is stronger than $|.|_{\mathbb{T}^n}$ (and, consequently, because the inverse is obvious, they are uniformly equivalent). Supposing the opposite, we get two sequences $\theta^{(s)},\widetilde{\theta}^{(s)} \in \mathbb{T}^n$, $s=1,2,\ldots$ such that
\begin{equation}
\label{DD-Polynomial-opposite}
\frac{\rho'\left(\theta^{(s)},\widetilde{\theta}^{(s)}\right)}{|\theta^{(s)} - \widetilde{\theta}^{(s)}|_{\mathbb{T}^n}} \leq \frac{1}{s}, \text{ for all } s=1,2\ldots.
\end{equation}
By the compactness of $\mathbb{T}^{n}$, we can assume that $\theta^{(s)}$ and $\widetilde{\theta}^{(s)}$ are convergent (to the same limit) and the fact that both metrics in (\ref{DD-Polynomial-opposite}) are translation invariant allows us to suppose that $\theta^{(s)}$ and $\widetilde{\theta}^{(s)}$ tend to zero. Now we fix $\theta_{0} \in \mathbb{T}^{n}$, let $\hat{\theta}^{(s)}:=\theta^{(s)} - \widetilde{\theta}^{(s)}$ and use the translation invariance again. From (\ref{DD-Polynomial-opposite}) we have
\begin{equation}
\frac{\rho'\left(\theta_{0},\theta_{0} + \hat{\theta}^{(s)}\right)}{|\hat{\theta}^{(s)}|_{\mathbb{T}^{n}}} \leq \frac{1}{s},  \text{ for all } s=1,2\ldots.
\end{equation}
One more observation is that we may assume that $|\hat{\theta}^{(s)}|_{\mathbb{T}^{n}} = |\hat{\theta}_{k_{0}}^{(s)}|_{\mathbb{T}^{1}}$ for some $1 \leq k_{0} \leq n$.

Now let $B_{1},\ldots,B_{m} \in E$ be a basis for the linear span of $A_{1},\ldots,A_{n}$ and
\begin{equation}
A_{k} = \sum\limits_{l=1}^{m}c^{(k)}_{l} B_{l}, 
\end{equation}
for some $c^{(k)}_{l} \in \mathbb{C}$. For the representing function of $P$ we have
\begin{equation}
h(t_{1},\ldots,t_{n}) = \sum\limits_{k=1}^{n}A_{k}e^{i2\pi t_{k}} = \sum\limits_{l=1}^{m} B_{l} P_{l}(t_{1},\ldots,t_{n}),
\end{equation}
where $P_{l}(t_{1},\ldots,t_{n}) = \sum\limits_{k=1}^{n}c^{(k)}_{l}e^{i2\pi t_{k}}$. Now we put $\theta_{0} = (\omega_{1} t_{0},\ldots,\omega_{n}t_{0})$ for some $t_0 \in \mathbb{R}$. From (\ref{FinalInducedMetric}) it is clear that $|h(\theta_{0})-h(\theta_{0} + \theta)|_{E} \leq \rho'(\theta_{0}, \theta_{0} + \theta)$, so
\begin{equation}
\label{DD-Polynomial-final-ineq}
\frac{|h(\theta_{0})-h(\theta_{0} + \hat{\theta}^{(s)})|_{E}}{|\hat{\theta}_{k_{0}}^{(s)}|_{\mathbb{T}^{1}}} \leq \frac{1}{s}, \text{ for all } s = 1,2\ldots.
\end{equation}
We may assume that the limit of $\frac{e^{i2\pi \hat{\theta}_{k}^{(s)}} - 1}{|\hat{\theta}_{k_{0}}^{(s)}|_{\mathbb{T}^{1}}}$ as $s \to +\infty$ exists and denote him by $\zeta_{k}$, for $k=1,\ldots,n$. Note that $|\zeta_{k}| \leq 2\pi$ and $\zeta_{k_{0}} = i 2\pi$. Thus, taking to the limit as $s \to +\infty$ in (\ref{DD-Polynomial-final-ineq}) we get
\begin{equation}
\sum\limits_{l=1}^{m}B_{l}\sum_{k=1}^{n}c^{(k)}_{l} \zeta_{k} e^{i2\pi \omega_{k} t_{0}} = 0.
\end{equation}
Using the linear independence of $B_{l}$'s we have $\sum_{k=1}^{n}c^{(k)}_{l} \zeta_{k} e^{i2\pi \omega_{k} t_{0}} = 0$, for $l = 1,\ldots, m$. Note that $t_{0}$ is chosen arbitrary, so by the uniqueness theorem for almost periodic functions the equality
\begin{equation}
 c^{(k)}_{l} \zeta_{k} = 0, \ \forall l=1,\ldots,m, \ \forall k=1,\ldots,n,
\end{equation}
must hold. We know that $\zeta_{k_{0}} = i 2\pi$ and, consequently, $c^{(k_{0})}_{l} = 0$ for all $l =1,\ldots, m$. The last is impossible because it was assumed in the initial statement that $A_{k_{0}} \not = 0$.
\end{proof}

\subsection{An upper estimate in the case of one irrational frequency}

We say, that a number $\omega \in \mathbb{R} \setminus \mathbb{Q}$ satisfies \textit{the Diophantine condition of order $\nu \geq 0$} if there is a constant $C>0$ such that for all natural $q$ and integer $p$ the inequality
\begin{equation}
	\left| \omega - \frac{p}{q} \right| \geq \frac{C}{q^{2+\nu}}.
\end{equation}
holds.

In the case of $\nu=0$ the number $\omega$ is called \textit{badly approximable}. Denote by $CD(\nu)$ the set of all irrational numbers satisfying the Diophantine condition of order $\nu$. The set $CD(\nu)$, where $\nu>0$, is a set of full Lebesgue measure (=its complement has measure zero). The set $CD(0)$, i.e. the set of badly approximable numbers, has measure zero (see \cite{Khinchin1997}), but it is still large enough (\cite{Schmidt1966}). The Diophantine condition can be expressed in terms of the continued fraction expansion. Here and further we use notations from subsection \ref{subsec: continuedfractions}.

\begin{proposition}
	\label{DCstatement}
	The following statements are equivalent.
	\begin{enumerate}
		\item $\omega \in CD(\nu)$;
		\item $q_{k+1}=O\left(q^{1+\nu}_{k}\right)$;
		\item $a_{k+1}=O\left(q^{\nu}_{k}\right)$;
	\end{enumerate}
\end{proposition}

In this subsection we present a generalization of the result of K. Naito from \cite{Naito1996}. We note, that he did some assumptions for the inverse frequency $\frac{1}{\omega}$, but it is more natural to do this for $\omega$ itself. We will show that it is equivalent to make such assumptions for $\omega$ or for $\frac{1}{\omega}$.

\begin{proposition}
	\label{InvertDCStatement}
	The following statements are equivalent.
	\begin{enumerate}
		\item $\omega \in CD(\nu)$;
		\item $\frac{1}{\omega} \in CD(\nu)$;
	\end{enumerate}
\end{proposition}

\begin{proof}
	Without loss of generality we suppose that $\omega \in (0,1)$. Let $\widetilde{\omega} = \frac{1}{\omega}$. So, $$\widetilde{\omega}_0 = \widetilde{\omega} - \left\lfloor \widetilde{\omega} \right\rfloor = G (\omega) = G (\omega_0) = \omega_{1}.$$
	Thus, $\widetilde{\omega}_{k}=\omega_{k+1}$ and $\widetilde{\omega}=\frac{1}{\omega}=[a_1;a_2,a_3,\ldots,]$, where $a_k$ is the term of continued fraction expansion of $\omega$. By proposition \ref{DCstatement}, we have that $\frac{1}{\omega} \in CD(\nu)$.
\end{proof}

We say that $\omega$ has the \textit{$G$-property} if there is a constant $C_{\omega}>1$, such that $q_{k+1} \geq C_{\omega} q_{k}$ for all $k=1,2,\ldots$.

\begin{proposition}
	\label{DecayStatement}
	The following statements are equivalent.
	\begin{enumerate}
		\item $\omega$ has the $G$-property;
		\item $\liminf\limits_{k\to +\infty}\frac{q_{k+1}}{q_{k}}>1$;
		\item For the sequence of all indexes $j_{k}$, such that $a_{j_{k}}=1$, the sequence $\{a_{j_{k} - 1}\}$ is bounded;
	\end{enumerate}
\end{proposition}
\begin{proof}
	The equivalence of $(1)$ and $(2)$ is obvious. Let $b_{k}=\frac{q_{k+1}}{q_{k}}$. Then, using (\ref{ConvEq2}), we have the equality
	\begin{equation}
		\label{geometricMonotonneLemma}
		b_{k}=a_{k+1} + \frac{q_{k-1}}{q_{k}}=a_{k+1}+\frac{1}{a_{k}+\frac{q_{k-2}}{q_{k-1}}}=a_{k+1} + \cfrac{1}{a_{k} + \cfrac{1}{a_{k-1} + \cfrac{1}{ \ddots + \cfrac{1}{a_1} }}}.
	\end{equation}
	From (\ref{geometricMonotonneLemma}) it is clear that if $a_{k+1}=1$ and $a_{k}$ is large enough then $b_k$ is close to 1 and vice versa. Thus, the statement is proved.
\end{proof}

An immediate corollary of proposition \ref{DecayStatement} and the proof of proposition \ref{InvertDCStatement} is the following. 

\begin{corollary}
	\label{KpropCorollary}
	An irrational number $\omega$ has the $G$-property if and only if $\frac{1}{\omega}$ has the $G$-property.
\end{corollary}

Now we are ready to prove an upper estimate of the Diophantine dimension of an $E$-quasiperiodic function $u(t)=h(\omega t, t)$, where $h$ satisfies\footnote{Here $C_{h},\varepsilon_{0}>0$ and $\alpha_{1},\alpha_{2} \in (0,1]$ are constants.}
\begin{equation}
\label{h:cond1}
|h(t_{1},s)-h(t_{2},s)|_{E} \leq C_{h} |t_{1} - t_{2}|^{\alpha_{1}}, \ \text{для всех $t_{1},t_{2},s \in \mathbb{R}$, \ $|t_{1}-t_{2}| \leq \varepsilon_{0}$},
\end{equation}
\begin{equation}
\label{h:cond2}
|h(t,s_{1})-h(t,s_{2})|_{E} \leq C_{h} |s_{1} - s_{2}|^{\alpha_{2}}, \ \text{для всех $t,s_{1},s_{2} \in \mathbb{R}$, \ $|s_{1}-s_{2}|\leq \varepsilon_{0}$}.
\end{equation}
This is a generalization of theorem 3 in \cite{Naito1996}. Since the main method remains the same we give a shortened proof.

\begin{theorem}
	\label{KnaitoOnefreqEstTheorem}
	Let $u=h(\omega t, t)$ be a quasiperiodic function, where $h$ satisfies (\ref{h:cond1}) and (\ref{h:cond2}). Suppose that $\omega \in \mathbb{R} \setminus \mathbb{Q}$ has the $G$-property and satisfies the Diophantine condition of order $\nu$. Let $\alpha=\max\{\alpha_1,\alpha_2\}$; then for some constant $K>0$ and all sufficiently small $\varepsilon>0$ the inequality
	\begin{equation}
		l_{u}(\varepsilon) \leq K \left(\frac{1}{\varepsilon}\right)^{\frac{1+\nu}{\alpha}}.
	\end{equation}
	holds. In particular, $\mathfrak{Di}(u) \leq \frac{1+\nu}{\alpha}$.
\end{theorem}

\begin{proof}
	Let $\widetilde{\omega}=\frac{1}{\omega}$. As it follows from proposition $\ref{InvertDCStatement}$ we have that $\widetilde{\omega} \in CD(\nu)$ and $\widetilde{\omega}$ has the $G$-property. For the $k$-th convergent $\frac{p_k}{q_k}$ of $\widetilde{\omega}$ we have
	\begin{equation}
		\label{FreqDiophantineApprox}
		\left|\widetilde{\omega}-\frac{p_k}{q_k}\right| < \frac{1}{q_{k+1}q_k}.
	\end{equation}
Following the proof of theorem 3 in \cite{Naito1996}, we can show that the value $L(\varepsilon_k)=q_{k+1}\widetilde{\omega}$ can be considered as an inclusion length for $\varepsilon_{k}$-almost periods, where
	\begin{equation}
		\label{NaitoEstimateEpsilonDef}
		\varepsilon_{k}=\frac{C_{h}}{1+C_{\widetilde{\omega}}^{-\alpha_2}} \left(\frac{1}{q_{k+1}}\right)^{\alpha_2}.
	\end{equation}
Now for all sufficiently small $\varepsilon>0$ let $k_{0}$ be such that $\varepsilon_{k_{0}+1} \leq \varepsilon < \varepsilon_{k_{0}}$. Then we can consider the value $L(\varepsilon) := L(\varepsilon_{k_{0}+1}) = q_{k_0+2}\widetilde{\omega}$ as an inclusion length for $\varepsilon$-almost periods\footnote{At this place there is a mistake in the proof of theorem 3 in \cite{Naito1996}. Here K. Naito puts $L(\varepsilon) := L(\varepsilon_{k_{0}})$, but this is obviously wrong. Because of such a mistake he did not assume anything similar to the Diophantine condition of the frequency. However, he has treating the case of badly approximable numbers ($\nu=0$), so this mistake does not affect his results.}.
 Let $C_{\nu}$ be a constant such that $q_{k+1} \leq C_{\nu}q^{1+\nu}_{k}$ (see item (2) of proposition \ref{DCstatement}). Then for some $K>0$ we have
	\begin{equation}
		\label{NaitoFirstFinalEstimate}
		L(\varepsilon)=L(\varepsilon_{k_0+1})=q_{k_0+2}\widetilde{\omega} \leq C_{\nu}\widetilde{\omega} q^{1+\nu}_{k_0+1} \leq K \left(\frac{1}{\varepsilon_{k_0}}\right)^{\frac{1+\nu}{\alpha_2}} \leq K \left(\frac{1}{\varepsilon}\right)^{\frac{1+\nu}{\alpha_2}}.
	\end{equation}
Analogously, we can prove a similar estimate for the case of H\"{o}lder exponent $\alpha_{1}$.
\end{proof}

Note that an estimate for the $n$-frequency case, i.e. $u(t)=h(t,\omega_{1}t,\ldots,\omega_{n-1}t)$, can be proved (theorem 4 in \cite{Naito1996} treat the case of simultaneously badly approximable frequencies).  But, as we know, it is hard to determine the simultaneous Diophantine condition of order $\nu>0$ for a given numbers $1,\omega_{1},\ldots,\omega_{n-1}$ if $n \geq 3$.

One can compare theorem \ref{KnaitoOnefreqEstTheorem} to previous results. It follows, that for many quasiperiodic functions the lower bound of the Diophantine dimension given by the theorem \ref{HullDimensionsTh} is reached for the case of simultaneously badly approximable frequencies. In terms of subsection \ref{subs:Liouvillephenomenon}, the distribution of such a trajectory is well-approximated in a very short time-interval by corresponding part of the trajectory.
\section{Diophantine dimension for almost periodic solutions of non-linear evolution equations with strongly monotone operator}

Let $H$ be a real Hilbert space, $V$ is a reflexive real Banach space and $V^{*}$ is the dual to $V$. We suppose that
\begin{equation}
V \mathop{\subset}\limits_{i} H \mathop{\subset}\limits_{i*} V^{*},
\end{equation}
where the inclusions are dense and continuous. Let $\|i\|=\gamma$ and denote by $\langle v_1,v_2\rangle$ the dual pair for $v_1 \in V^{*}$ and $v_2 \in V$).

Suppose that for almost all $t \in \mathbb{R}$ the operator $A(t) \colon V \to V^{*}$ is bounded from $V$ to $V^{*}$ and satisfies \textit{strong monotonicity} condition, i.e. for some $M>0$ and $\alpha>1$  the inequality\footnote{In original work K. Naito treats only to the case $\alpha=2$. It is easy to generalize his method for $\alpha>1$.}
\begin{equation}
\label{strongmonotonicity}
\langle A(t)u - A(t)v,u-v \rangle \geq  M\|u-v\|^{\alpha}_{H}.
\end{equation}
holds for all $u,v \in V$.

Assume that there exists $u \in L^{\infty}(\mathbb{R};H) \cap L_{loc}^{2}(\mathbb{R};V)$, such that $u' \in L_{loc}^{2}(\mathbb{R};V^{*})$ and
\begin{equation}
\label{ODEinBanachSpace}
\frac{du}{dt} + A(t)u=f(t), \ \text{for almost all $t \in \mathbb{R}$}.
\end{equation}

\begin{theorem}[K. Naito, theorem 1 in \cite{Naito1982}]
	\label{KNaitoTheorem}
	Let $f \colon \mathbb{R} \to V^{*}$ be an $V^{*}$-almost periodic function and $\tau$ is an $\varepsilon$-almost period for $f$. Suppose also
	\begin{equation}
	\sup\limits_{t \in \mathbb{R}}\|A(t)u(t)-A(t+\tau)u(t)\|_{V^{*}} \leq \kappa.
	\end{equation}
	Then for some $C=C(\gamma,\alpha,M)>0$ the number $\tau$ is a $C(\varepsilon+\kappa)^{\frac{1}{\mu-1}}$-almost period\footnote{I.e. the inequality analogous to (\ref{almostperiod}) is satisfied.} of $u$, where $u$ is considered as a function from $\mathbb{R}$ to $H$.
\end{theorem}

\begin{corollary}
	\label{coroll:KNAITO}
	Within the assumptions of theorem \ref{KNaitoTheorem} suppose that $u \in C(\mathbb{R};H)$ and $A(t)$ is independent of $t$, i.e. $A(t) \equiv A$; then $u$ is $H$-almost periodic and
	\begin{equation}
	\mathfrak{Di}(u) \leq (\alpha-1)\mathfrak{Di}(f).
	\end{equation}
\end{corollary}

Note that condition (\ref{strongmonotonicity}) is often used to prove the existence of bounded and almost periodic solutions to certain evolution equations (see \cite{Pankov2012,Pankov1984}). Besides theorem \ref{KNaitoTheorem}, the strong monotonicity condition is used to estimate other asymptotic properties of solutions, for example, the rate of decay (see \cite{Zauzua1988}).
\section{Some remarks}
Determining exact values or bounds of the Diophantine dimension of an almost periodic function one may face a problem when this value becomes infinite\footnote{Similar effects appears in the dimension theory for dynamical systems with multiple time (see \cite{AnikushinReitmann2016})}. This may happen if there are infinitely many rationally independent Fourier exponents or the exponents can be approximated extremely good. In these cases one can determine the following values
\begin{equation}
	\mathfrak{Di}(u,d):=\limsup_{\varepsilon \to 0+}\frac{\ln l_{u}(\varepsilon)}{\left(\ln 1 / \varepsilon\right)^{d}},
\end{equation}
where $d>0$, or choose a more suitable function in the denominator.

There are so called effective versions of the Kronecker theorem, where an upper estimate on the solution $t_{\varepsilon}$ of (\ref{KroneckerSystem}) is given (see \cite{Vorselen2010} and refs within). In this way corollary \ref{PolynomialLowerEstimate} may be interest in the number theory. The problem to find $\varepsilon$-almost periods of a polynomial $P(t)=\sum\limits_{k=1}^{n}A_{k} e^{i 2\pi \omega_{k} t}$ is equivalent (in some sense) to finding the solutions $\tau_{\delta}$ of the system
\begin{equation}
\label{effectiveKroneckerSys}
|\omega_{j}\tau| < \delta \pmod 1, \  j=1,2,\ldots,n,
\end{equation}
where $\delta$ is proportional to $\varepsilon$ and, consequently, the inclusion length is well defined for the set of solutions of (\ref{effectiveKroneckerSys}) and it has the same asymptotic as the inclusion length for almost periods of $P(t)$. It is always hard to get lower bounds (not only in this case, but for many problems of mathematics), so it seems interesting.

The shift operator defined on the hull $\mathcal{H}(u)$ of an almost periodic function $u$ defines an uniquely ergodic almost periodic dynamical system. The exponents of such a flow are the same as for $u$. Therefore, effects discussed in subsection \ref{subs:Liouvillephenomenon} may appear in this case.

\bibliographystyle{amsplain}

\end{document}